\documentclass[12pt]{amsart}
\usepackage{amsmath,amssymb,amsthm,graphics,amscd,mathrsfs}
\usepackage{amscd, amsfonts, amssymb, amsthm}
\usepackage{parskip, fullpage, verbatim}
\usepackage[colorlinks, breaklinks, linkcolor=blue]{hyperref} 
\usepackage{breakurl}
\usepackage{array}
\usepackage{latexsym}
\usepackage{enumerate}

\newcommand\CA{{\mathscr A}}

\renewcommand\CD{{\mathscr D}}

\newcommand\CS{{\mathcal S}}

\newcommand\RR{{\mathscr R}}

\newcommand\BBR{{\mathbb R}}
\newcommand\BBZ{{\mathbb Z}}

\newcommand {\Sage}{\textsf{SAGE}}

\newcommand\codim{\operatorname{codim}}

\newcommand\rk{\operatorname{rk}}

\numberwithin{equation}{section}

\theoremstyle{plain}
\newtheorem{lemma}[equation]{Lemma}
\newtheorem{theorem}[equation]{Theorem}

\theoremstyle{definition}
\newtheorem{defn}[equation]{Definition}
\newtheorem{remark}[equation]{Remark}
\newtheorem{remarks}[equation]{Remarks}


\begin{document}

\title[Counting chambers in restricted Coxeter arrangements]
{Counting chambers in restricted Coxeter arrangements}

\author[T. M\"oller]{Tilman M\"oller}
\address
{Fakult\"at f\"ur Mathematik,
Ruhr-Universit\"at Bochum,
D-44780 Bochum, Germany}
\email{tilman.moeller@rub.de}

\author[G. R\"ohrle]{Gerhard R\"ohrle}
\email{gerhard.roehrle@rub.de}


\allowdisplaybreaks

\begin{abstract}
Solomon showed that 
the Poincar\'e polynomial of a Coxeter group $W$
satisfies a product 
decomposition depending on the exponents of $W$.
This polynomial coincides with the 
rank-generating function of the poset of regions of
the underlying Coxeter arrangement.
In this note we determine all instances when 
the analogous factorization property of 
the rank-generating function of the poset of 
regions holds for a restriction of a
Coxeter arrangement.
It turns out that this is always the case 
with the exception of some instances in type $E_8$.
\end{abstract}

\maketitle


\section{Introduction}

Much of the motivation 
for the study of arrangements 
of hyperplanes comes 
from Coxeter arrangements. 
They consist of the reflecting hyperplanes 
associated with the 
reflections of the underlying Coxeter group.
Solomon showed that 
the Poincar\'e polynomial $W(t)$
of a Coxeter group $W$ satisfies a product 
decomposition depending on the exponents of $W$, see \eqref{eq:solomon}.
This polynomial coincides with the 
rank-generating function of the poset of regions of
the underlying Coxeter arrangement, see \S \ref{s:rankgenerating}.
The aim of this note is to 
classify all cases when 
the analogous factorization property of 
the rank-generating function of the poset of 
regions holds for an arbitrary restriction of a
Coxeter arrangement.
It turns out that this is always the case 
with the exception of some instances in type $E_8$,
see Theorem \ref{thm:main}.

The analogous factorization property for 
a localization of a Coxeter arrangement
is an immediate consequence of Solomon's theorem
and a theorem of Steinberg \cite[Thm.~1.5]{steinberg:invariants}, 
see Remark \ref{rems:thmmain}(iv).

\subsection{The Poincar\'e polynomial of a Coxeter group}
\label{ssec:coxeter}
Let $(W,S)$ be a Coxeter group with a distinguished set of generators, $S$,   
see \cite{bourbaki:groupes}. Let $\ell$ be the length function 
of $W$ with respect to $S$.
The \emph{Poincar\'e polynomial} $W(t)$ of 
the Coxeter group $W$ is the polynomial in $\BBZ[t]$ defined by 
\begin{equation}
\label{eq:poncarecoxeter}
W(t) := \sum_{w \in W} t^{\ell(w)}.
\end{equation}

The following factorization of 
$W(t)$ is due to Solomon \cite{solomon:chevalley}:
\begin{equation}
\label{eq:solomon}
W(t) = \prod_{i=1}^n(1 + t + \ldots + t^{e_i}),
\end{equation}
where $\{e_1, \ldots, e_n\}$ is the 
set of exponents of $W$.
See also Macdonald \cite{macdonald:coxeter}.

\subsection{The rank-generating function of the posets of regions}
\label{s:rankgenerating}

Let $\CA = (\CA,V)$ be a 
hyperplane arrangement in the real vector space $V=\BBR^n$. 
A \emph{region} of $\CA$ is a connected component of the 
complement $V \setminus \cup_{H \in \CA}H$ of $\CA$.
Let $\RR := \RR(\CA)$ be the set of regions of $\CA$.
For $R, R' \in \RR$, we let $\CS(R,R')$ denote the 
set of hyperplanes in $\CA$ separating $R$ and $R'$.
Then with respect to a choice of a fixed 
base region $B$ in $\RR$, we can partially order
$\RR$ as follows:
\[
R \le R' \quad \text{ if } \quad \CS(B,R) \subseteq \CS(B,R').
\]
Endowed with this partial order, we call $\RR$ the
\emph{poset of regions of $\CA$ (with respect to $B$)} and denote it by
$P(\CA, B)$. This is a ranked poset of finite rank,
where $\rk(R) := |\CS(B,R)|$, for $R$ a region of $\CA$, 
\cite[Prop.~1.1]{edelman:regions}.
The \emph{rank-generating function} of $P(\CA, B)$ is 
defined to be the following polynomial in 
$\BBZ[t]$
\begin{equation*}
\label{eq:rankgen}
\zeta(P(\CA,B), t) := \sum_{R \in \RR}t^{\rk(R)}. 
\end{equation*}

\bigskip
Let $W = (W,S)$ be a Coxeter group with associated reflection arrangement 
$\CA = \CA(W)$ which consists of the reflecting hyperplanes of 
the reflections in $W$ in the real space $V=\BBR^n$, where $|S| = n$. 
Note that the Poincar\'e polynomial $W(t)$
associated with $W$ 
given in \eqref{eq:poncarecoxeter} 
coincides with the rank-generating function of the poset of regions of 
the underlying reflection arrangement 
$\CA(W)$ with respect to $B$ being the dominant Weyl chamber of $W$
in $V$; 
see \cite{bjoerneredelmanziegler} or \cite{jambuparis:factored}.

\bigskip

Thanks to work of  Bj\"orner, Edelman, and Ziegler 
\cite[Thm.~4.4]{bjoerneredelmanziegler}
(see also Paris \cite{paris:counting}), respectively  
Jambu and Paris \cite[Prop.~3.4, Thm.~6.1]{jambuparis:factored},
in case of a real arrangement $\CA$  
which is supersolvable (see see \S \ref{ssect:supersolv}), 
respectively inductively factored (see \S \ref{ssect:factored}), 
there always exists a suitable base region $B$ so that 
$\zeta(P(\CA,B), t)$
admits a multiplicative decomposition which 
is equivalent to \eqref{eq:solomon}
determined by the exponents of $\CA$, 
see Theorem \ref{thm:mult-zeta}.

\subsection{Restricted Coxeter arrangements}
\label{s:restrictedCoxeter}

Let $W$ be a Coxeter group with reflection arrangement 
$\CA = \CA(W)$ in $V=\BBR^n$.
We consider the following generalization of the 
Poincar\'e polynomial $W(t)$ of $W$.
Let $X$ be in the intersection lattice $L(\CA)$ of $\CA$,
i.e.~$X$ is the subspace in $V$ given by the intersection 
of some hyperplanes in $\CA$.
Then we can consider the restricted arrangement 
$\CA^X$ which is the induced arrangement in $X$ from $\CA$, 
see \S \ref{ssect:arrangements}. 
In a case-by-case study, 
Orlik and Terao showed in \cite{orlikterao:free} that the restricted 
arrangement 
$\CA^X$ is always free, so we can speak of the exponents of $\CA^X$,
see \cite[\S 4]{orlikterao:arrangements}.
In case $W$ is a Weyl group, Douglass \cite[Cor.~6.1]{douglass:adjoint}
gave a uniform proof
of this fact by means of an elegant, conceptual Lie theoretic argument.

It follows from the discussion above that in the special instances
when either $\CA^X$ is supersolvable 
(which is for instance always the case for $X$ of dimension at most $2$)
or inductively factored, or else if $X$ is just the ambient space $V$
(so that $\CA^V = \CA)$, then 
$\zeta(P(\CA^X,B), t)$ is known to factor analogous to 
\eqref{eq:solomon} involving the exponents of $\CA^X$. 

Fadell and Neuwirth \cite{fadellneuwirth}
showed that the braid arrangement 
is fiber type and Brieskorn \cite{brieskorn:tresses} 
proved this for the reflection arrangement of the
hyperoctahedral group. This property is equivalent to 
being supersolvable, see \cite{terao:modular}. 
Therefore, since any restriction of a supersolvable 
arrangement is again supersolvable, \cite{stanley:super}, in
case of the symmetric or hyperoctahedral group $W$, 
we see that $\CA(W)^X$ is supersolvable for any $X$.
Thus in each of these cases the rank 
generating function of the poset of regions of
 $\CA(W)^X$ factors as in \eqref{eq:solomon},
thanks to Theorem \ref{thm:mult-zeta}.

Therefore, 
it is natural to study the rank-generating function 
of the poset of regions of an arbitrary restriction of a
Coxeter arrangement. 
The following gives a complete classification of all instances
when $\zeta(P(\CA^X,B), t)$ factors analogous to 
\eqref{eq:solomon}.

\begin{theorem}
\label{thm:main}
Let $W$ be a finite, irreducible Coxeter group with 
reflection arrangement $\CA = \CA(W)$.
Let $\CA^X$ be the restricted arrangement 
associated with $X \in L(\CA)\setminus\{V\}$.
Then there is a suitable choice of a base region $B$ so that 
the rank-generating function of the poset of regions of $\CA^X$
satisfies the multiplicative formula
\begin{equation}
\label{eq:poinprod}
\zeta(P(\CA^X,B), t) = \prod_{i=1}^n (1 + t + \ldots + t^{e_i}),
\end{equation}
where $\{e_1, \ldots, e_n\}$ is the 
set of exponents of $\CA^X$ 
if and only if one of the following holds:
\begin{itemize}
\item[(i)] $W$ is not of type $E_8$;
\item[(ii)] $W$ is of type $E_8$ and either the rank of $X$ is at most $3$,  
but $\CA^X \not \cong (E_8,A_2A_3)$ and $\CA^X \not \cong (E_8,A_1A_4)$,
or else $\CA^X \cong (E_8,D_4)$.
\end{itemize}
\end{theorem}

We prove Theorem \ref{thm:main} in Section \ref{sec:proof}.
For classical $W$, either $\CA(W)^X$ is supersolvable and 
the result follows from Theorem \ref{thm:mult-zeta},  
or else $W$ is of type $D$ and 
$\CA(W)^X$ belongs to a particular family of arrangements
$\CD_p^k$ for $0 \le k \le p$ studied by Jambu and Terao, 
\cite[Ex.~2.6]{jambuterao:free}.
We prove Theorem \ref{thm:main} for the family
$\CD_p^k$ in Lemma \ref{lem:dn}.

For $W$ of exceptional type, 
there are $31$ restrictions $\CA(W)^X$ of rank at least $3$
(up to isomorphism)
that need to be considered. These are handled by 
computational means, see Remark \ref{rem:exc}.

\begin{remarks}
\label{rems:thmmain}
(i). 
In the statement of the theorem and later on
we use the convention to label the $W$-orbit 
of $X \in L(\CA)$ by the Dynkin type $T$ 
of the stabilizer $W_X$ of $X$ in $W$ which 
is itself a Coxeter group, 
by Steinberg's theorem \cite[Thm.~1.5]{steinberg:invariants}.
So we denote the restriction $\CA^X$ just by the pair
$(W,T)$; see also \cite[App.~C, D]{orlikterao:arrangements}.

(ii).
Among the restrictions $\CA(W)^X$ all  
supersolvable and all inductively factored instances are known, 
see Theorems \ref{thm:super-restriction} and 
\ref{thm:nice-restriction} below.
Thus, by Theorem \ref{thm:mult-zeta}, in each of these cases 
$\zeta(P(\CA^X,B), t)$
factors as in \eqref{eq:poinprod}.

(iii).
Hoge checked that the exceptional case  
$(E_8,A_2A_3)$ from Theorem \ref{thm:main}
is isomorphic to 
the real simplicial arrangement 
``$A_4(17)$'' from Gr\"unbaum's list \cite{gruenbaum}.
It was observed by Terao that the latter does 
not satisfy the product rule \eqref{eq:poinprod}, 
\cite[p.~277]{bjoerneredelmanziegler}.
It is rather remarkable that this arrangement
makes an appearance as a restricted Coxeter arrangement.  
In contrast, according to Theorem \ref{thm:main}, 
the rank-generating function of the poset of regions
of $(E_8,A_1^2A_3)$ does factor according to
\eqref{eq:poinprod}. In particular, 
these two arrangements are not isomorphic, as 
claimed erroneously in \cite[App.~D]{orlikterao:arrangements}.

(iv).
For  $X$ in $L(\CA(W))$ consider  
the localization $\CA(W)_X$ of $\CA(W)$ at $X$, 
which consists of all members of $\CA(W)$ containing $X$,
see \S \ref{ssect:arrangements}.
Then, since 
the stabilizer $W_X$  in $W$ of $X$
is itself a Coxeter group,
by Steinberg's theorem \cite[Thm.~1.5]{steinberg:invariants}, 
and since $\CA(W)_X = \CA(W_X)$, 
by \cite[Cor.~6.28(2)]{orlikterao:arrangements}, 
it follows from Solomon's factorization \eqref{eq:solomon}
that the rank generating function of the poset of regions of
$\CA(W)_X$ 
(with respect to the base chamber being the unique chamber of 
$\CA(W)_X$ containing the dominant Weyl chamber of $W$) factors 
analogous to \eqref{eq:solomon} involving 
the exponents of $W_X$.

(v).
In Lie theoretic terms, 
for $W$ a Weyl group, 
$W(t^2)$ is the Poincar\'e polynomial 
of the flag variety of a semisimple
linear algebraic group with Weyl group $W$.
The  formula \eqref{eq:solomon} then gives a well-known 
factorization of the Poincar\'e polynomial of the flag variety.

If $W$ is of type $A$ or $B$, then 
each restriction $\CA(W)^X$ is
the Coxeter arrangement of the same Dynkin type of 
smaller rank, 
cf.~\cite[Props.~6.73, 6.77]{orlikterao:arrangements}.
Thus, by the previous paragraph, in these instances, 
$\zeta(P(\CA^X,B), t^2)$ 
is just the Poincar\'e polynomial  
of the flag variety of a semisimple 
linear algebraic group of the same Dynkin type as $W$ but of smaller rank.

In view of these examples, it is natural to wonder whether in general 
there is a suitable projective variety 
associated with a fixed semisimple group $G$ with Weyl group $W$
whose Poincar\'e polynomial is related to 
the rank-generating function of the poset of 
regions for any restriction of $\CA(W)$ in the same manner 
as in these special instances above, relating to 
and generalizing the flag variety of $G$.
\end{remarks}

For general information about arrangements and Coxeter groups,
we refer the reader to \cite{bourbaki:groupes} and 
\cite{orlikterao:arrangements}.

\section{Recollections and Preliminaries}
\label{sect:prelims}

\subsection{Hyperplane arrangements}
\label{ssect:arrangements}
Let $V = \BBR^n$ be an $n$-dimensional real vector space.
A \emph{(real) hyperplane arrangement} $\CA = (\CA, V)$ in $V$ 
is a finite collection of hyperplanes in $V$ each 
containing the origin of $V$.
We denote the empty arrangement in $V$ by $\Phi_n$.

The \emph{lattice} $L(\CA)$ of $\CA$ is the set of subspaces of $V$ of
the form $H_1\cap \ldots \cap H_i$ where $\{ H_1, \ldots, H_i\}$ is a subset
of $\CA$. 
For $X \in L(\CA)$, we have two associated arrangements, 
firstly
$\CA_X :=\{H \in \CA \mid X \subseteq H\} \subseteq \CA$,
the \emph{localization of $\CA$ at $X$}, 
and secondly, 
the \emph{restriction of $\CA$ to $X$}, $\CA^X = (\CA^X,X)$, where 
$\CA^X := \{ X \cap H \mid H \in \CA \setminus \CA_X\}$.
Note that $V$ belongs to $L(\CA)$
as the intersection of the empty 
collection of hyperplanes and $\CA^V = \CA$. 
The lattice $L(\CA)$ is a partially ordered set by reverse inclusion:
$X \le Y$ provided $Y \subseteq X$ for $X,Y \in L(\CA)$.

Throughout, we only consider arrangements $\CA$
such that $0 \in H$ for each $H$ in $\CA$.
These are called \emph{central}.
In that case the \emph{center} 
$T(\CA) := \cap_{H \in \CA} H$ of $\CA$ is the unique
maximal element in $L(\CA)$  with respect
to the partial order.
A \emph{rank} function on $L(\CA)$
is given by $r(X) := \codim_V(X)$.
The \emph{rank} of $\CA$ 
is defined as $r(\CA) := r(T(\CA))$.

\subsection{Free arrangements}
\label{ssect:free}

Free arrangements play a fundamental role in the 
theory of hyperplane arrangements, 
see \cite[\S 4]{orlikterao:arrangements} for the definition and
properties of this notion. Crucial for our purpose is the fact that 
associated with a free arrangement is a set of important invariants, its 
(multi)set of \emph{exponents}, denoted by $\exp \CA$.

\subsection{Supersolvable arrangements}
\label{ssect:supersolv}

We say that $X \in L(\CA)$ is \emph{modular}
provided $X + Y \in L(\CA)$ for every $Y \in L(\CA)$,
\cite[Cor.~2.26]{orlikterao:arrangements}.

\begin{defn}
[{\cite{stanley:super}}]
\label{def:super}
Let $\CA$ be a central arrangement of rank $r$.
We say that $\CA$ is 
\emph{supersolvable} 
provided there is a maximal chain
\[
V = X_0 < X_0 < \ldots < X_{r-1} < X_r = T(\CA)
\]
 of modular elements $X_i$ in $L(\CA)$,
cf.~\cite[Def.~2.32]{orlikterao:arrangements}.
\end{defn}

\noindent Note that arrangements of rank at most $2$ are always supersolvable, 
e.g.~see \cite[Prop.~4.29(iv)]{orlikterao:arrangements} and 
supersolvable arrangements are always free, e.g.~see 
\cite[Thm.~4.58]{orlikterao:arrangements}.
Also, restrictions of a supersolvable 
arrangement are again supersolvable, \cite[Prop.~3.2]{stanley:super}.

\subsection{Nice and inductively factored arrangements}
\label{ssect:factored}

The notion of a \emph{nice} or \emph{factored} 
arrangement is due to Terao \cite{terao:factored}.
It generalizes the concept of a supersolvable arrangement, e.g.~see
\cite[Prop.~2.67, Thm.~3.81]{orlikterao:arrangements}.
Terao's main motivation was to give a 
general combinatorial framework to 
deduce tensor factorizations of the underlying Orlik-Solomon algebra,
see also \cite[\S 3.3]{orlikterao:arrangements}.
We refer to \cite{terao:factored} for the relevant 
notions and properties
(cf.~ \cite[\S 2.3]{orlikterao:arrangements}).

There is an  analogue of Terao's 
Addition Deletion Theorem for 
free arrangements 
(\cite[Thm.~4.51]{orlikterao:arrangements}) 
for the class of 
nice arrangements, see \cite[Thm.~3.5]{hogeroehrle:factored}.
In analogy to the case of free arrangements, this motivates
the notion of an 
\emph{inductively factored} arrangement, 
see \cite{jambuparis:factored}, \cite[Def.~3.8]{hogeroehrle:factored}
for further details on this concept.

The connection with the previous notions is as follows.
Supersolvable arrangements are always inductively factored 
(\cite[Prop.~3.11]{hogeroehrle:factored})
and inductively factored arrangements are always free
(\cite[Prop.~2.2]{jambuparis:factored}) so
that we can talk about the exponents of such arrangements.

The following theorem due to Jambu and Paris, 
\cite[Prop.~3.4, Thm.~6.1]{jambuparis:factored},
was first shown by Bj\"orner, Edelman and Ziegler
for $\CA$ supersolvable in \cite[Thm.~4.4]{bjoerneredelmanziegler}
(see also Paris \cite{paris:counting}).

\begin{theorem}
\label{thm:mult-zeta}
If $\CA$ is inductively factored, then
there is a suitable choice of a base region $B$ so that 
$\zeta(P(\CA,B), t)$ satisfies the multiplicative formula
\begin{equation}
\label{eq:poinprod2}
\zeta(P(\CA,B), t) = \prod_{i=1}^n (1 + t + \ldots + t^{e_i}),
\end{equation}
where $\{e_1, \ldots, e_n\} = \exp \CA$ is the 
set of exponents of $\CA$.
\end{theorem}

\subsection{Restricted root systems}
\label{ssect:restrictedroots}
Given a root system for $W$,
associated with a member $X$ from $L(\CA(W))$ we have a 
\emph{restricted root system} which consists of the restrictions
of the roots of $W$ to $X$, see \cite[\S 2]{brundangoodwin:grading}.
As in the absolute case, bases of the restricted root system correspond
bijectively to chambers of the arrangement $\CA(W)^X$,
\cite[Cor.~7]{brundangoodwin:grading}. 
More specifically, let $\Phi$ be a root system for $W$ 
and let $\Delta\subset \Phi$ 
be a set of simple roots. 
In view of Remark \ref{rems:thmmain}(i), 
choosing $X\in L(\CA(W))$ amounts to specifying the Dynkin type $T$ 
of the parabolic subgroup $W_X$, so that the pair
$(W,T)$ characterizes $A(W)^X$. 
Let $\mathcal{B}_T$ be the set of all subsets of $\Delta$ 
that generate a root system of Dynkin type $T$.
Fixing an element $\Delta_J\in \mathcal{B}_T$, 
the bases for $\Phi$ containing $\Delta_J$ are 
in bijective correspondence with 
the bases for the restricted root system, 
\cite[Thm.~10]{brundangoodwin:grading}. 

Furthermore, the set $\mathcal{B}_T$ characterizes a set of 
representatives for the action of the \emph{restricted Weyl group} 
on the set of chambers of the arrangement $\CA(W)^X$, 
\cite[Lem.~11]{brundangoodwin:grading}.
Thus there is a suitable choice of a base region $B$
such that $\zeta(P(\CA(W)^X, B),t)$ factors according to \eqref{eq:poinprod}, 
if and only if there is 
such a choice among regions that arise from elements in $\mathcal{B}_T$. 

\section{Proof of Theorem \ref{thm:main}}
\label{sec:proof}

It is well known that 
if $W$ is of type $A$ or $B$, then 
the Coxeter arrangement $\CA(W)$ is supersolvable and 
so is every restriction thereof. 
So Theorem \ref{thm:main} follows in this 
case from Theorem \ref{thm:mult-zeta}.
Therefore, for $W$ of classical type, we only need to consider 
restrictions for $W$ of type $D$.
The restrictions  
$\CD_p^k$ for $0 \le k \le p$ 
of Coxeter arrangements
of type $D$ are given by 
the defining polynomial
\[
Q(\CD_p^k) := x_{p-k+1} \cdots x_p \prod_{1 \le i < j \le p}(x_i^2 - x_j^2),
\]
see \cite[Ex.~2.6]{jambuterao:free} (\cite[Cor.~6.86]{orlikterao:arrangements}).

In view of Theorem \ref{thm:mult-zeta}, we next recall 
the relevant parts of the classifications of the 
supersolvable and inductively factored restrictions of reflection 
arrangements from \cite{amendhogeroehrle:super} and  
\cite{moellerroehrle:nice}, respectively.
Here we focus on such $X$ in $L(\CA)$ of dimension at least $3$, 
as a restriction to a smaller dimensional member of $L(\CA)$ 
is already supersolvable.

\begin{theorem}
[{\cite[Thm.~1.3]{amendhogeroehrle:super}}]
\label{thm:super-restriction}
Let $W$ be a finite, irreducible Coxeter group 
with reflection arrangement 
$\CA = \CA(W)$ and let $X \in L(\CA)\setminus\{V\}$ 
with $\dim X \ge 3$.
Then the restricted arrangement 
$\CA^X$ is supersolvable
if and only if 
one of the following holds: 
\begin{itemize}
\item[(i)] $\CA$ is of type $A$ or of type $B$, or
\item[(ii)]  $W$ is of type $D_n$ for $n \ge 4$ and 
$\CA^X \cong \CD^k_p$, where $p = \dim X$ and $p - 1 \leq k \leq p$;
\item[(iii)] $\CA^X$ is $(E_6,A_3)$, $(E_7, D_4)$, $(E_7, A_2^2)$, or $(E_8, A_5)$.
\end{itemize}
\end{theorem}

As noted above, every supersolvable restriction from 
Theorem \ref{thm:super-restriction}
is inductively factored.

\begin{theorem}
[{\cite[Thms.~1.5, 1.6]{moellerroehrle:nice}}]
\label{thm:nice-restriction}
Let $W$ be a finite, irreducible Coxeter group 
with reflection arrangement 
$\CA = \CA(W)$ and let $X \in L(\CA)\setminus\{V\}$ 
with $\dim X \ge 3$.
Then the restricted arrangement 
$\CA^X$ is inductively factored 
if and only if 
one of the following holds: 
\begin{itemize}
\item[(i)] $\CA^X$ is supersolvable, or 
\item[(ii)] $W$ is of type $D_n$ for $n \ge 4$ and 
$\CA^X \cong \CD^{p-2}_p$, where $p = \dim X$;
\item[(iii)] 
$\CA^X$ is one of $(E_6, A_1A_2), (E_7, A_4)$, or $(E_7, (A_1A_3)'')$.
\end{itemize}
\end{theorem}

It follows from Theorem \ref{thm:mult-zeta} that in all instances
covered in  Theorem \ref{thm:nice-restriction}, 
$\zeta(P(\CA,B), t)$ satisfies the factorization property of 
\eqref{eq:poinprod2} with respect to a suitable choice of 
base region $B$.
In particular, Theorem \ref{thm:main} holds in all these instances.

It is not apparent that the rank-generating function 
of the poset of regions of $\CD_p^k$ factors 
according to \eqref{eq:poinprod} for $1\leq k \leq p-3$.
For, these arrangements are neither reflection arrangements 
nor are they inductively factored, by the results above. 
To show that 
the factorization property from \eqref{eq:poinprod} 
also holds in these instances, 
we first parameterize the regions $\RR(\CD_p^k)$ suitably 
and then prove a recursive formula for $\zeta(P(\CD_p^k,B),t)$.

\begin{remark}
\label{rem:regionsDlk}
Since the inequalities given by the hyperplanes 
do not change within a region, 
the set of regions is uniquely determined 
by specifying one interior point for each region. Let 
$$M_p^k:= \left\{(x_1,\dots,x_p)\in \{\pm 1,\dots,\pm p\}^p 
\mid x_1,\dots,x_{p-k}\neq -1,\hspace{0.4em} |x_i| \neq |x_j| 
\hspace{0.4em} \forall i\neq j  \right\}.$$
It is easy to verify that each region in 
$\RR := \RR(\CD_p^k)$ contains 
exactly one element of $M_p^k$. 
So this gives a parametrization for the regions in $\RR$. 
Without further comment, we frequently identify points in $M_p^k$ with their 
respective regions in  $\RR$.
For $x\in M_p^k$, write $R_x\in \RR$ for the unique region containing $x$.
Once a base region $B$ in $\RR$ 
is chosen so that $\RR$ becomes a ranked poset,  
we may write
$$\zeta(P(\CD_p^k,B),t) = \sum\limits_{x\in M_p^k} t^{\rk(R_x)}. $$

Using this notation it is easy to see which regions are 
adjacent and which hyperplanes are walls of a given region. 
Let $x=(x_1,\dots,x_p)\in M_p^k$. 
If $x_j = x_i \pm 1$, 
then $\ker(x_i-x_j)$ is a wall of $R_x$ 
and the corresponding adjacent region 
is obtained from $x$ by exchanging 
$x_i$ and $x_j$ in $x$. If $x_j = -(x_i\pm 1)$, 
then $\ker(x_i+x_j)$ is a wall of $R_x$ 
and the adjacent region again originates 
from $x$ by exchanging 
$x_i$ and $x_j$ \emph{but maintaining their respective signs}.
Finally, if $x_i=\pm 1$ and $p-k<i\leq p$, then $\ker(x_i)$ 
is a wall of $R_x$ and the adjacent region
is obtained by exchanging $x_i$ with $-x_i$.
\end{remark}

For our subsequent results, 
we choose $B_p := R_y\in \RR$ for $y=(p,p-1,\dots,1)$ 
as our base chamber independent of $k$. 

\begin{lemma}
\label{lem:dnhilf}
Let $p\geq 3$, $k\in\{0,\dots,p\}$ and $B_p\in \RR$ as above. 
For an arbitrary $i\in\{1,\dots,p\}$, we have
\begin{equation}
\label{eq:p}
\sum_{\substack{x\in M_p^k \\ x_i = p }} t^{\rk(R_x)} = 
\begin{cases}
t^{i-1} \cdot \zeta(P(\CD_{p-1}^{k  },B_{p-1}),t) &\text{ if } i\leq p-k, \\
t^{i-1} \cdot \zeta(P(\CD_{p-1}^{k-1},B_{p-1}),t) &\text{ if } i > p-k,
\end{cases}\\
\end{equation}
and
\begin{equation}
\label{eq:-p}
\sum\limits_{\substack{x\in M_p^k \\ x_i = -p }} t^{\rk(R_x)} = 
\begin{cases}
t^{2p-i-1} \cdot \zeta(P(\CD_{p-1}^{k},B_{p-1}),t)   &\text{ if } i\leq p-k, \\
t^{2p - i} \cdot \zeta(P(\CD_{p-1}^{k-1},B_{p-1}),t) &\text{ if } i > p-k.
\end{cases}
\end{equation}
\end{lemma}

\begin{proof}
Set $N^- := \{x\in M_p^k\mid x_i = -p\}$.
Thanks to Remark \ref{rem:regionsDlk}, 
no hyperplane involving the coordinate $x_i$ 
lies between any two regions of $N^-$. 
Setting
\[
z = (z_1,\dots,z_i,\dots,z_p) := 
(p-1,p-2, \ldots,p-i+1, -p,p-i-1, \ldots,2, 1)\in N^-,
\] 
there are only hyperplanes involving $x_i$ between $B_p$ and $R_z$. 
More precisely,  we have
\[
\CS(B_p,R_z)=
\begin{cases}
\{\ker(x_i - x_j)\mid j\leq i\}  \cup \{\ker( x_i \pm x_j )\mid i<j \leq p) \} &\text{ for } i\leq p-k, \\
\{\ker(x_i - x_j)\mid j\leq i\}  \cup \{\ker( x_i \pm x_j )\mid i<j \leq p) \} \cup \{\ker(x_i)\} &\text{ for } i>p-k.
\end{cases}
\] 
So if we choose an arbitrary $x\in N^-$, we have
$$\CS(B_p,R_x) = \CS(B_p,R_z) \mathbin{\dot\cup} \CS(R_z,R_x).$$
Consequently, we obtain
\begin{equation}
\label{eq:rk}
\rk(R_x) = |\CS(B_p,R_z)| + |\CS(R_z,R_x)| = 
\begin{cases}
t^{2p-i-1} + |\CS(R_z,R_x)|  &\text{ for } i\leq p-k, \\
t^{2p - i} + |\CS(R_z,R_x)|  &\text{ for } i>p-k.
\end{cases}
\end{equation}
Now set
\begin{equation}
\label{eq:A}
\CA := 
\begin{cases}
\CD_{p-1}^{k}   &\text{ if } i\leq p-k, \\
\CD_{p-1}^{k-1} &\text{ if } i > p-k,
\end{cases}
\end{equation}
and identify the set of regions $\RR(\CA)$ 
of $\CA$ with the corresponding set of 
$(p-1)$-tuples as in Remark \ref{rem:regionsDlk}. 
Then simply omitting the $i$-th coordinate defines a map
\[
h:N^- \longrightarrow \RR(\CA)
\]
which is bijective, $h(R_z)=B_{p-1}$ 
and if $\widetilde\rk$ denotes the rank function on $P(\CA,B_{p-1})$, 
then we get $|\CS(R_z,R_x)| = \widetilde\rk(h(R_x))$.
Therefore, by \eqref{eq:rk}, \eqref{eq:A}
and the bijectivity of $h$, we get
\begin{align*}
\sum\limits_{\substack{x\in M_p^k \\ x_i = -p }} t^{\rk(R_x)} 
&= t^{|\CS(B_p,R_z)|} \sum\limits_{x\in N^-} t^{|\CS(R_z,R_x)|} \\
&= t^{|\CS(B_p,R_z)|} \sum\limits_{x\in N^-} t^{\widetilde\rk(h(R_x))} \\
&= t^{|\CS(B_p,R_z)|} \sum\limits_{x\in \RR(\CA)} t^{\widetilde\rk(R_x)} \\
&= t^{|\CS(B_p,R_z)|} \zeta(P(\CA,B_{p-1}),t) \\
& =
\begin{cases}
t^{2p-i-1} \cdot \zeta(P(\CD_{p-1}^{k},B_{p-1}),t)   &\text{ if } i\leq p-k, \\
t^{2p - i} \cdot \zeta(P(\CD_{p-1}^{k-1},B_{p-1}),t) &\text{ if } i > p-k.
\end{cases}
\end{align*}
So \eqref{eq:p} follows.

Next let $N^+:=\{x\in M_p^k\mid x_i = p\}$ and set 
\[
z = (z_1,\dots,z_i,\dots,z_p) := 
(p-1,p-2, \ldots,p-i+1,p,p-i-1, \ldots,2, 1)\in N^+.
\]
Then $\CS(B_p,R_z)=\{\ker(x_i - x_j) \mid 1\leq j < i\}$ has cardinality $i-1$. The proof of this case is similar to the one above, 
and is left to the reader.
So \eqref{eq:-p} follows.
\end{proof}

The next technical lemma is needed in the proof of Lemma \ref{lem:dn}.
For ease of notation, we set
\[
F(e_1,\dots,e_m) := \prod_{i=1}^{m} (1+t+\dots+t^{e_i}) \in \BBZ[t] 
\]
for any $m\geq 1$ and integers $e_1,\dots,e_m \geq 1$.
In particular, $F(e) = 1+t+\dots+t^e$.
Also note that for $j>0$, we have 
\begin{equation}
\label{eq:F}
F(j-1)(1+t^j) = F(2j-1).
\end{equation} 

\begin{lemma}
\label{lem:delta}
Let $p\geq 3$ and $0\leq k\leq p$. Define
\[
\Delta_p^k := 
\sum_{i=1}^{p-k} (t^{i-1}+t^{2p-i-1}) F(p+k-2) 
+ \sum_{i=p-k+1}^{p} (t^{i-1}+t^{2p-i}) F(p+k-3).
\]
Then
\[
\Delta_p^k
= F(p+k-1,2p-3).
\]
\end{lemma}

\begin{proof}
We argue by induction on $k$. First let $k=0$. Then, using \eqref{eq:F}, we have
\begin{align*}
\Delta_p^0 &= \sum_{i=1}^{p} (t^{i-1}+t^{2p-i-1}) F(p-2) \\
&= (1+\dots+t^{p-1})F(p-2) + (t^{p-1}+\dots+t^{2p-2}) F(p-2) \\
&= F(p-1,p-2) + t^{p-1}F(p-1,p-2) \\
&= F(p-1,p-2)\left(1+t^{p-1}\right) \\
&= F(p-1,2p-3).
\end{align*}
Now let $k>0$ and assume that the statement is true for $k'<k$. Then using
the inductive hypothesis, we get
\begin{align*}
\Delta_p^k &= \Delta_p^{k-1} + t^{p+k-2}\sum_{i=1}^{p-k}\left(t^{i-1}+t^{2p-i-1}\right)
+ t^{p+k-3}\sum_{i=p-k+1}^{p}\left(t^{i-1}+t^{2p-i}\right) \\
&\hspace{1em} - F(p+k-3)\left(t^{p-k}+t^{p+k-2}\right) 
+ F(p+k-4)\left(t^{p-k}+t^{p+k-1}\right) \\
&= F(2p-3,p+k-2)+(t^{p+k-2}+\dots+t^{2p-3}+t^{2p+2k-3}+\dots+t^{3p+k-4}) \\
&\hspace{1em} + (t^{2p-3}+\dots+t^{2p+k-4}+t^{2p+k-3}+\dots+t^{2p+2k-4}) - (t^{2p-3}+t^{p+k-2}) \\
&= F(2p-3,p+k-2) + t^{p+k-1}(1+\dots+t^{2p-3}) \\
&= F(2p-3)(F(p+k-2)+t^{p+k-1}) \\
& = F(2p-3,p+k-1),
\end{align*}
as claimed.
\end{proof}

Finally, armed with Lemmas \ref{lem:dnhilf} and \ref{lem:delta}, we are able to 
prove the desired result for the arrangements $\CD_p^k$.

\begin{lemma}
\label{lem:dn}
The rank-generating function
of the poset of regions of $\CD_p^k$ factors according to 
\eqref{eq:poinprod} for all $1 \le k \le p-3$ and $p \ge 4$.
\end{lemma}

\begin{proof}
We argue by induction on $n=p+k$. For $n=3$, the result holds vacuously. 
So let $1 \le k \le p-3$ and $p \ge 4$ and assume that for all $p'$, $k'$, 
with $1 \le k' \le p'-3$, $p' \ge 4$ and $n > p'+k'$,  
the arrangement $\CD_{p'}^{k'}$ satisfies \eqref{eq:poinprod}. 
Note that 
\begin{equation}
\label{eq:expD}
\exp(\CD_p^k) = \exp(\CD_{p-1}^{p-1}) \cup \{p+k-1\},
\end{equation}
see \cite[Ex.~2.6]{jambuterao:free}.
Then the inductive hypothesis together 
with Lemmas \ref{lem:dnhilf} and \ref{lem:delta} 
and \eqref{eq:expD} imply
\begin{align*}
&\zeta(P(\CD_p^k,B_p),t)
=\sum_{x\in M_p^k} t^{\rk(R_x)} 
=\sum_{i=1}^{p} \sum_{\substack{x\in M_p^k \\ x_i = \pm p}} t^{\rk(R_x)} \\
&=\sum_{i=1}^{p-k} (t^{i-1}+t^{2p-i-1}) \zeta(P(\CD_{p-1}^{k},B_{p-1}),t) 
+ \sum_{i=p-k+1}^{p} (t^{i-1}+t^{2p-i}) \zeta(P(\CD_{p-1}^{k-1},B_{p-1}),t) \\
&= \sum_{i=1}^{p-k} (t^{i-1}+t^{2p-i-1}) F\left(\exp(\CD_{p-1}^{k}) \right)
+ \sum_{i=p-k+1}^{p} (t^{i-1}+t^{2p-i}) F\left( \exp(\CD_{p-1}^{k-1}) \right) \\
&= F\left(\exp(\CD_{p-2}^{p-2})\right) \left( \sum_{i=1}^{p-k} (t^{i-1}+t^{2p-i-1}) F(p+k-2) 
+ \sum_{i=p-k+1}^{p} (t^{i-1}+t^{2p-i}) F(p+k-3) \right) \\
&= F\left(\exp(\CD_{p-2}^{p-2})\right)  \Delta_p^k  \\
& = F\left(\exp(\CD_{p-2}^{p-2})\right) F(2p-3,p+k-1) \\
&= F\left(\exp(\CD_{p}^{k})\right).
\end{align*}
This completes the proof of the lemma.
\end{proof}

\begin{remark}
\label{rem:exc}
In view of Theorems \ref{thm:mult-zeta}, 
\ref{thm:super-restriction} and
\ref{thm:nice-restriction}, 
Lemma \ref{lem:dn} settles all the remaining 
classical instances of Theorem \ref{thm:main}.
It follows from 
Theorems \ref{thm:super-restriction} and
\ref{thm:nice-restriction}
that there are $31$ instances for $W$ of exceptional type to 
be checked (here we take the isomorphisms of rank $3$ 
restrictions $\CA(W)^X$ into account, 
cf.~\cite[App.~D]{orlikterao:arrangements}).
We have verified that 
$\zeta(P(\CA(W)^X,B), t)$ satisfies 
the factorization property \eqref{eq:poinprod}
precisely in all the instances when $W$ is of exceptional type, as
specified in Theorem \ref{thm:main}.
In the listed exceptions, $\zeta(P(\CA(W)^X,B), t)$ does not factor 
according to this rule with respect to any choice of base region.
This was checked using the computer algebra package \Sage, \cite{sage}.

We used the \Sage-package 
\emph{HyperplaneArrangements} which provides
methods to compute 
$\zeta(P(\CA, B), t))$ for given $\CA$ and $B$. 
More specifically, the algorithm is initiated  
with a list containing the vector space $V$ as a polytope and 
for each hyperplane in $\CA$ splits each polytope in the 
current list into two polytopes, 
defined by a positive resp.~negative inequality, 
while discarding all empty solutions. This results in a list of chambers 
implemented as polytopes. 
After specifying a base region $B$ 
the algorithm checks for each region $R$ 
and each hyperplane $H$ whether $H$ separates $B$ from $R$.

In addition, we used the results from \cite[\S 2]{brundangoodwin:grading}, as
detailed in Section \ref{ssect:restrictedroots} 
to greatly reduce the number of chambers that have to be tested. 
This method worked for all exceptional restrictions other than 
$(E_8,A_1)$, as the latter is simply too big for \Sage\  to compute 
all its chambers at once. 
For this case we instead used the 
bijective correspondences recalled in 
\ref{ssect:restrictedroots} to compute 
the chambers directly from the elements of 
the Weyl group $W(E_8)$.
By ordering the group elements by length using 
a depth-first search algorithm implemented 
in the \Sage-package \emph{ReflectionGroup}, we were able to compute 
the chambers of the restricted arrangement ordered by rank, 
so we could conclude that the rank-generating polynomial 
of the poset of regions for the restriction
$\CA^X = (E_8,A_1)$ does not factor
according to \eqref{eq:poinprod} 
after computing only a small portion of the entire polynomial 
$\zeta(P(\CA^X, B), t))$.
\end{remark}


\bigskip 
{\bf Acknowledgments}: 
We are grateful to T.~Hoge for checking that 
 the simplicial arrangement 
``$A_4(17)$'' from Gr\"unbaum's list 
coincides with the restriction $(E_8,A_2 A_3)$.
We would also like to thank C.~Stump for 
helpful discussions concerning computations in \Sage.

The research of this work was supported by 
DFG-grant RO 1072/16-1.



\bibliographystyle{amsalpha}

\newcommand{\etalchar}[1]{$^{#1}$}
\providecommand{\bysame}{\leavevmode\hbox to3em{\hrulefill}\thinspace}
\providecommand{\MR}{\relax\ifhmode\unskip\space\fi MR }
\providecommand{\MRhref}[2]{%
  \href{http://www.ams.org/mathscinet-getitem?mr=#1}{#2} }
\providecommand{\href}[2]{#2}


\end{document}